\documentclass[12pt,reqno]{amsart}
\usepackage{amsmath,amsthm,amscd,amsfonts,amssymb,color}
\usepackage{cite}
\usepackage[mathscr]{eucal}
\usepackage[bookmarksnumbered,colorlinks,plainpages]{hyperref}
\newtheorem{theorem}{Theorem}[section]

\newtheorem{corollary}[theorem]{Corollary}
\newtheorem{open question}[theorem]{Open Question}
\theoremstyle{definition}
\newtheorem{definition}[theorem]{Definition}
\newtheorem{example}[theorem]{Example}
\theoremstyle{remark}
\newtheorem{remark}[theorem]{Remark}
\numberwithin{equation}{section}
\def\DJ{\leavevmode\setbox0=\hbox{D}\kern0pt\rlap
 {\kern.04em\raise.188\ht0\hbox{-}}D}

\begin{document}
\title{The contractive principle for mappings in $b_v(s)$-metric spaces}
%\title[An application to a thermostat model via fixed point results in Banach spaces]{An application to a thermostat model via fixed point results in Banach spaces}
%\title[Altering distance functions in ... thermostat model]{Altering distance functions in fixed point results with application to thermostat model}
%\title[Positive solution via fixed point results]{Positive solution to fractional thermostat model in Banach spaces via fixed point results}
\author[H.\ Garai, L.K. \ Dey, P.\ Mondal]
{Hiranmoy Garai$^{1}$,  Lakshmi Kanta Dey$^{2}$, Pratikshan Mondal$^{3}$}
%
%%\thanks{}
\address{{$^{1}$\,} Department of Mathematics,
                    National Institute of Technology
                    Durgapur, India.}
                    \email{hiran.garai24@gmail.com}
\address{{$^{2}$\,} Department of Mathematics,
                    National Institute of Technology
                    Durgapur, India.}
                    \email{lakshmikdey@yahoo.co.in}
\address{{$^{3}$\,} Department of Mathematics,                                         					Durgapur Government College, 
                    Durgapur, India.}
                    \email{real.analysis77@gmail.com}
%                    
%
%\subjclass{$47H10$, $54H25$.}
%\keywords{Altering distance function, fixed point, Banach space, double sequence, thermostat model.}

\subjclass[2010]{$47$H$10$, $54$H$25$.}
\begin{abstract}
In this article, we introduce the notions of sequentially compactness and boundedly compactness in the framework of a newly defined $b_v(s)$-metric space which is a generalization of usual metric spaces and several other abstract spaces. We establish correlations   between sequentially compactness and boundedly compactness. Moreover, we prove some fixed point results of contractive  mapping  in this setting, from which we can deduce several analogous   fixed point results. Finally, we illustrate  some non-trivial examples  to validate the significances and motivations of this manuscript. 
\end{abstract}
 
\maketitle
%
%\setcounter{page}{1}
%
%\centerline{}
%
%\centerline{}
%
\section{ Introduction}
There are a few extensions of the notion of a metric space in the literature which were attempted to generalize several known fixed point results in usual metric spaces. One of these is the concept of $b$-metric spaces which was introduced by Czerwik \cite{C4} in $1993$. Afterwards, in the year of $2000$, Branciari \cite{B2} coined the notion of rectangular metric spaces or generalized metric spaces by making a slight modification in the triangle inequality of the usual metric spaces. As a generalization of $b$-metric spaces and rectangular metric spaces, the authors in \cite{GRRS} introduced the concept of  rectangular $b$-metric spaces. Taking into account all these concepts, many mathematicians have elaborated analogous fixed point results in these settings which meliorated and improved the original fixed point theories in usual metric spaces. In the literature, there is a huge amount of relevant texts available for intent readers (see \cite{C6,ADKR,HRPA,KS,SPK,DIRV,BKGK,EKS,DD, GRRS,RPKH,AKL,KR1} and the references therein). 
In $2000$, Branciari \cite{B2} instigated the concept of $v$-generalized metric spaces.

After all such generalizations, Mitrovi\'c and Radenovi\'c \cite{MR} introduced the concept of $b_v(s)$-metric spaces which generalize all the metric spaces discussed above and defined such kind of metric spaces as follows: 
\begin{definition} \cite{MR}
Let $X$ be a non-empty set and let $d:X \times X \to \mathbb{R}$ be a function and $v\in \mathbb{N}$. Then $(X,d)$ is said to be a $b_v(s)$-metric space if for all $x,y\in X$ and for all distinct points $u_1,u_2, \hdots, u_v \in X$, each of them different from $x$ and $y$, the following conditions hold:
\begin{itemize}
\item[{(i)}] $d(x,y) \geq 0$ and $d(x,y)=0$ if and only if $x=y$;
\item[{(ii)}] $d(x,y) = d(y,x)$;
\item[{(iii)}] there exists a real number $s \geq 1$ such that $$d(x,y)\leq s\left[d(x,u_1) + d(u_1,u_2) + \hdots + d(u_v,y)\right].$$
\end{itemize} 
\end{definition}
The authors of \cite{MR} also achieved some interesting fixed point results associated to Banach and Reich contractions.

In this sequel, we recall the definition of  contractive mapping, i.e.,  a self-mapping $T$ on a metric space $(X,d)$ satisfying $$d(Tx,Ty) < d(x,y)$$ for all $x,y \in X$ with $x\neq y$. It may be noted that completeness of the underlying space $X$ does not give the guaranty of existence of fixed point of this map, but if $X$ is compact, it gives the guaranty \cite{E1}.  One of our main motivations in this paper is  to find some (mild) additional criteria on the underlying $b_v(s)$-metric space $X$, which give the existence of fixed point. To proceed in this direction, we introduce the notions of sequentially compactness and boundedly compactness in $b_v(s)$-metric spaces and establish correlations   between them. In such spaces, we prove some fixed point theorems related to contractive mapping which improve and generalize some standard fixed point results due to Edelstein \cite{E1} and Suzuki \cite{S2}.

\section{ Preliminaries}
In $1961$, Edelstein \cite{E1} proved that a contractive self-mapping on a compact metric space has a unique fixed point. One can easily verify that, in this result the compactness of $X$ can not be replaced by completeness. Therefore, some additional conditions have to be imposed on $X$ or on $T$ together with the completeness of $X$ to ensure the existence of fixed point of $T$. Many researchers tried to find such additional conditions.  One of them is given by the following theorem:
%\begin{theorem}\cite{E1} \label{t1}
%Let $(X,d)$ be a compact metric space and $T: X \to X$ be  a contractive self-mapping, i.e., $$d(Tx,Ty) < d(x,y)$$ for all $x,y \in X$ with $x\neq y$. Then there exists a unique fixed point of $T$.
%\end{theorem}
%But it can be easily verified that if the compactness condition of $X$ in Theorem \ref{t1} is reduced to completeness, then $T$ may be fixed point free. So we need some additional assumptions with the completeness of $(X,d)$ by which we can find a fixed point $T$. Many researchers tried to find some additional assumptions on $T$ to establish the assurance of fixed point of $T$. One of them is given by the following theorem:
\begin{theorem} \cite{C5}
Let $(X,d)$ be a complete metric space and let $T:X \to X$ be a mapping such that $$d(Tx,Ty) < d(x,y)$$ for all $x,y \in X$ with $x\neq y$ and let the following hold:

$(A)$ For any $\epsilon > 0$, there exists $\delta > 0$ such that $$d(x, y) < \epsilon + \delta \Rightarrow d(Tx,Ty) \leq \epsilon$$ for any $x, y \in X$.

Then $T$ has a unique fixed point $z$ and for any $x\in X$, the sequence of iterates $\{T^nx\}$ converges to $z$.
\end{theorem}
Suzuki introduced another additional weaker assumption with the completeness of $(X,d)$ to assure fixed point of $T$ in the following theorem:
\begin{theorem} \cite{S2}
Let $(X,d)$ be a complete metric space and let $T:X \to X$ be a mapping such that $$d(Tx,Ty) < d(x,y)$$ for all $x,y \in X$ with $x\neq y$ and let the following hold:

$(B)$ For any $x\in X$ and for any $\epsilon > 0$, there exists $\delta > 0$ such that $$d(T^ix,T^jx) < \epsilon + \delta \Rightarrow d(T^{i+1}x,T^{j+1}x)\leq \epsilon$$ for any $i,j \in \mathbb{N} \cup \{0\}$.

Then $T$ has a unique fixed point $z$ and for any $x\in X$, the sequence of iterates $\{T^nx\}$ converges to $z$.
\end{theorem}

Mitrovi\'c and Radenovi\'c introduced the notions of Cauchy sequences and completeness in $b_v(s)$-metric space.
\begin{definition} \cite{MR}
Let $(X,d)$ be a $b_v(s)$-metric space, $\{x_n\}$ be a sequence in $X$ and $x\in X$. Then, the following are defined as follows:
\begin{itemize}
\item[{(i)}] The sequence $\{x_n\}$ is said to be a Cauchy sequence, if for any $\epsilon>0$ there exists $N\in \mathbb{N}$ such that $$d(x_n,x_{n+p})< \epsilon$$ for all $n>N$ and for all $p \in \mathbb{N}$ or equivalently, if $$\displaystyle \lim_{n\to \infty}d(x_n,x_{n+p})=0$$ for all $p \in \mathbb{N}$.
\item[{(ii)}] The sequence $\{x_n\}$ is said to be convergent to $x$,  if for any $\epsilon>0$ there exists $N\in \mathbb{N}$ such that $$d(x_n,x)< \epsilon $$ for all $n>N$ and this fact is represented by $$\displaystyle \lim_{n\to \infty}x_n =x ~~\mbox{or}~~ x_n \to x ~~\mbox{as}~~ n\to \infty.$$
\item[{(iii)}] $(X,d)$ is said to be complete $b_v(s)$-metric space if every Cauchy sequence in $X$ converges to some $x\in X$.
\end{itemize}
\end{definition}
In this manuscript, we introduce the concepts of sequentially compactness and boundedly compactness of a $b_v(s)$-metric space.
\begin{definition}
Let $(X,d)$ be a $b_v(s)$-metric space and $\{x_n\}$ be a sequence in $X$. Then the sequence $\{x_n\}$ is said to be a bounded sequence if there exists a real number $M>0$ such that $$d(x_n,x_m) \leq M ~~\mbox{for all}~~ n,m\in \mathbb{N}$$ or equivalently $$d(x_n,x_{n+k}) \leq M ~~\mbox{for all}~~ n,k\in \mathbb{N}.$$
\end{definition}
\begin{definition}
Let $(X,d)$ be a $b_v(s)$-metric space, then $X$ is said to be sequentially compact if every sequence $\{x_n\}$ in $X$ has a subsequence which converges to some point of $X$.
\end{definition}

\begin{definition}
Let $(X,d)$ be a $b_v(s)$-metric space, then $X$ is said to be boundedly compact if every bounded sequence $\{x_n\}$ in $X$ has a subsequence which converges to some point of $X$.
\end{definition}
Clearly it follows from definition that, every boundedly compact $b_v(s)$-metric space is sequentially compact but not conversely. To investigate this we frame the following example:
\begin{example}
Consider the set $X=X_1 \cup X_2$ where $X_1=\{\frac{1}{n}:n\in \mathbb{N} ~~\mbox{and}~~ n\geq 2\}$ and $X_2=\{0,1,2\}$. We define a function $d:X\times X \to \mathbb{R}$ by
$$d(x,y)=\begin{cases}
|n-m|, \text{if}~~ x,y \in X_1 ~~ \mbox{and}~~ x=\frac{1}{n}, y=\frac{1}{m}~~ \mbox{and}~~ |n-m| \neq 1,3;\\
\frac{1}{2}, \text{if}~~ x,y \in X_1 ~~ \mbox{and}~~ x=\frac{1}{n}, y=\frac{1}{m}~~ \mbox{and}~~ |n-m| = 1,3;\\
n, \text{if}~~ x\in X_1, y\in X_2 ~~ \mbox{and}~~ x=\frac{1}{n} ~~\text{or}~~ y\in X_1, x\in X_2 ~~ \mbox{and}~~ y=\frac{1}{n};\\
5, \text{if}~~  x,y \in X_2 ~~ \mbox{and}~~ x\neq y;\\
0, \text{if}~~  x,y \in X_2 ~~ \mbox{and}~~ x= y.
\end{cases}$$
Then, it is an easy task to verify that $(X,d)$ is a $b_3(2)$-metric space.

Note that the sequence $\{\frac{1}{n+2}\}$ has no subsequence which converges to some point of $X$. So, $(X,d)$ is not sequentially compact. But, one can easily check that, a sequence $\{x_n\}$ in $X$ is bounded if and only if the range of the sequence $\{x_n\}$ is finite. Thus every bounded sequence in $X$  has a  subsequence which converges to some point of  $X$, i.e, $(X,d)$ is boundedly compact.
\end{example}
%In this article, our aim is to improve and generalize the fixed point results associated with contractive self-mappings in sequentially compact and boundedly compact $b_v(s)$-metric spaces and also in complete metric spaces by introducing some weaker additional assumptions.
 \section{ Fixed Point Results}
At the beginning of this section, we prove a fixed point theorem related to contractive mappings in the structure of sequentially compact $b_v(s)$-metric spaces.
\begin{theorem} \label{th1}
Let $(X,d)$ be a sequentially compact $b_v(s)$-metric space and $T:X\to X$ be a mapping such that $$d(Tx,Ty) < d(x,y)$$ for all $x,y \in X$ with $x \neq y$. Then $T$ has a unique fixed point and for any $x\in X$ the sequence  $\{T^nx\}$ converges to that fixed point.
\end{theorem}
\begin{proof}
Let $x_0 \in X$ be arbitrary but fixed. We consider the sequence $\{x_n\}$ where $x_n=T^n x_0$ for each natural number $n$.

If $x_n = x_{n+1}$ for some $n\in \mathbb{N}$, then clearly $T$ has a fixed point. So we assume that $x_n \neq x_{n+1}$ for all $n\in \mathbb{N}$.

In this case we claim that all terms of $\{x_n\}$ are distinct. To prove our claim, we presume that  $x_n = x_m$ for some natural numbers  $n,m$ with $m>n$. Then $Tx_n = Tx_m$ i.e. $x_{n+1} = x_{m+1}.$ Therefore, 
\begin{align*}
d(x_{n+1},x_n) & =  d(x_{m+1},x_m)\\
& <  d(x_m,x_{m-1})\\
&  \hdots\\
&  \hdots\\
& <  d(x_{n+1},x_n)
\end{align*}
which is a contradiction. This proves our claim.

Now, since  $(X,d)$ is sequentially compact, so the sequence $\{x_n\}$ must have a convergent subsequence, say $\{x_{n_k}\}$ and let it converges to $z$, i.e., $\displaystyle \lim_{k\to \infty} x_{n_k} = z.$ Next, we show that $z$ is a fixed point of $T$. If $x_{n_k} =z$  for finitely many $n_k$, then we can exclude those $x_{n_k}$ from $\{x_{n_k}\}$ and  assume that $x_{n_k} \neq z$ for all $n_k$. If $x_{n_k} =z$  for infinitely many $n_k$, then $x_{n_k+1}=Tz$ for infinitely many $n_k$. So, $\{x_{n_k}\}$ contains a subsequence which converges to $Tz$, so we must have $z=Tz$, i.e., $z$ is a fixed point of $T$.

If $x_{n_k} =Tz$  for finitely many $n_k$, then we can exclude those $x_{n_k}$ from $\{x_{n_k}\}$ and  assume that $x_{n_k} \neq Tz$ for all $n_k$. If $x_{n_k} =Tz$  for infinitely many $n_k$, then  $\{x_{n_k}\}$ contains a subsequence which converges to $Tz$, so we must have $z=Tz$, i.e., $z$ is a fixed point of $T$.

Finally, if $x_{n_k} \neq z,Tz$ for all $n_k \in \mathbb{N}$. Then,
\begin{align} \label{eqn1}
d(z,Tz) &\leq  s\{d(z,x_{n_k +1}) + d(x_{n_k + 1},x_{n_k + 2}) + d(x_{n_k + 2},x_{n_k + 3}) + \dots  + \nonumber\\ & d(x_{n_k + v - 2},x_{n_k + v - 1}) + d(x_{n_k + v - 1},x_{n_k + v}) + d(x_{n_k + v},Tz)\}\nonumber\\
& <  s\{d(z,x_{n_k +1}) + d(x_{n_k + 1},x_{n_k + 2}) + d(x_{n_k + 2},x_{n_k + 3}) + \dots  + \nonumber\\ 
& d(x_{n_k + v - 2},x_{n_k + v - 1}) + d(x_{n_k + v - 1},x_{n_k + v}) + d(x_{n_k + v-1},z)\}.
\end{align}
Since $\displaystyle \lim_{k\to \infty} d(x_{n_k},z)=0$ and  $\displaystyle \lim_{k\to \infty} d(x_{n_k},x_{n_k+1})=0$, so taking limit as $k\to \infty$ in both sides of the Equation \ref{eqn1}, we get,
$$d(z,Tz)\leq 0 \Rightarrow d(z,Tz)=0.$$ This shows that $Tz=z$, i.e, $z$ is a fixed point of $T$. Hence combining all possible cases, we conclude that $z$ is a fixed point of $T$.

Now, we examine the uniqueness of this fixed point. To do this, let $z_1$ be another fixed point of $T$. Then we have 
\begin{align*}
d(z,z_1) &= d(Tz,Tz_1)\\
&< d(z,z_1),
\end{align*}
which is a contradiction. Hence, $z$ is the only fixed point of $T$.

Finally, we prove  that $\{x_n\}$ converges to $z$. If $x_n =z$ for some $n\in \mathbb{N}$, then clearly $\{x_n\}$ converges to $z$. Let us assume that  $x_n \neq z$ for all $n \in \mathbb{N}$. Since, $\{x_n\}$ contains a subsequence $\{x_{n_k}\}$ which converges to $z$, so $z$ is a cluster point of the sequence $\{x_n\}$. Let $z'$ be another cluster point of $\{x_n\}$, then $\{x_n\}$ contains a subsequence which converges to $z'$. Then by similar arguments as above we can show that $z'$ is a fixed point of $T$ and this will again lead to a contradiction. Henceforth, $z$ is the only cluster point of $\{x_n\}$.

Now, consider the sequence $\{s_n\}$ of real numbers where $s_n = d(x_n,z)$ for all $n\in \mathbb{N}$. Therefore
$$s_{n_k} = d(x_{n_k},z) \to 0 \mbox{  as  } n\to \infty.$$
This shows that the sequence $\{s_n\}$ contains the subsequence $\{s_{n_k}\}$ which converges to $0$ and so $0$ is a cluster point of  $\{s_n\}$.

Now, 
\begin{align*}
s_{n+1} & = d(x_{n+1},z)\\
&= d(Tx_n,Tz)\\
&< d(x_n,z) = s_n
\end{align*}
Consequently, $\{s_n\}$ is a decreasing sequence of non-negative real numbers and hence convergent. But $0$ is a cluster point of the convergent sequence $\{s_n\}$, so $\{s_n\}$ must converge to $0$. Therefore
$$s_n \to 0 \mbox{  as  } n\to \infty $$
$$ \Rightarrow d(x_n,z) \to 0 \mbox {  as  }  n\to \infty. $$
Thus, $\{x_n\}$ converges to $z$, i.e., $\{T^nx_0\}$ converges to $z$.

Since $x_0 \in X$ was arbitrary, it follows that $\{T^nx\}$ converges to the fixed point $z$ for any $x\in X$.
\end{proof}
From Theorem \ref{th1} we can deduce the following corollary by taking $s=1$ and $v=1$.
%If we take $s=1$ and $v=1$ in Theorem \ref{th1}, we get the following corollary (this is the main result in \cite{E1}):
\begin{corollary} \label{cr1}
Let $(X,d)$ be a compact metric space and $T:X\to X$ be mapping such that $$d(Tx,Ty)<d(x,y)$$ for all $x,y \in X$ with $x\neq y$. Then $T$ has a unique fixed point $z$ and for any $x\in X$ the sequence $\{T^nx\}$ converges to $z$.
\end{corollary}
\begin{remark}
Corollary \ref{cr1} extends Remark $3.1$ of \cite{E1}.
\end{remark}
%\begin{corollary}
%If we take $v=1 \mbox{  and   } s=1$ then we get the following corollary;
%\end{corollary}
In the above theorem sequentially compactness condition cannot be replaced by boundedly compactness of the space which follows from the following example.
\begin{example} \label{e1}
Let $X=\mathbb{N}$ and define a function $d:X \times X \to \mathbb{R}$ by
$$d(x,y)=
\begin{cases}
0, ~~\text{if}~~ x=y;\\
10|x-y|, ~~\text{if}~~ x,y<10;\\
\frac{|x-y|}{10}, ~~\text{if}~~ x,y \geq 10;\\
5, ~~\text{otherwise}. 
\end{cases}$$ 
Then it is easy to see  that, $(X,d)$ is a $b_2(1000)$-metric space. Also, one can easily check that the $b_2(1000)$-metric space $(X,d)$ is boundedly compact but not sequentially compact.

Next we consider a function $T:X\to X$ defined by
$$Tx=
\begin{cases}
x+20, ~~\text{if}~~ x\leq 10;\\
10, ~~\text{if}~~x>10.
\end{cases}$$
Let $x,y\in X$ be such that $x\neq y$. Let us consider the following cases:

Case I: Let, $x,y \leq 10$, then $$Tx=x+20 ~~\mbox{and}~~ Ty=y+20.$$
Therefore, 
\begin{align*}
d(Tx,Ty)&= \frac{|(x+20)-(y+20)|}{10}\\
&= \frac{|x-y|}{10}\\
\end{align*} whereas 
$$d(x,y)=
\begin{cases}
10|x-y|, ~~\text{if}~~ x,y<10;\\
5, ~~\text{if either $x=10$ or $y=10$}.
\end{cases}$$
So, $d(Tx,Ty) < d(x,y)$.

Case II: Let, $x,y > 10$, then $$Tx=10 ~~\mbox{and}~~ Ty=10.$$
So, it is obvious that $d(Tx,Ty) < d(x,y)$.

Case III: Let, $x>10$ and $y\leq 10$, then $$Tx=10 ~~\mbox{and}~~ Ty=y+20.$$
Therefore, 
\begin{align*}
d(Tx,Ty)&= \frac{|10-(y+20)|}{10}\\
&= \frac{y+10}{10}\\
& \leq  2.
\end{align*} and
$$d(x,y)=5.$$
Thus it follows that  $d(Tx,Ty) < d(x,y)$.
Also, from the formulation of $T$ it is clear that $T$ has no fixed point.
\end{example}
%Hence, we see that if we replace the sequentially compactness of the $b_v(s)$-metric space $(X,d)$ in Theorem \ref{th1} by the weaker concept boundedly compactness, then the result may not hold.

Therefore we are in search of an additional condition either on X or on T with the boundedly compactness of X to get a unique fixed point of T. Here in the next theorem we deal with one of such additional condition.
\begin{theorem} \label{th2}
Let $(X,d)$ be a boundedly compact $b_v(s)$-metric space and $T:X\to X$ be a mapping such that $$d(Tx,Ty) < d(x,y)$$ for all $x,y \in X$ with $x \neq y$. Also assume that, for any $x \in X$ and for any $k\in \mathbb{N}$ with $k\geq v$ there exists a real number $M>0$ (depending on $x$) such that $d(x,T^{k-v}x) \leq M$.   Then $T$ has a unique fixed point and for any $x\in X$ the sequence  $\{T^nx\}$ converges to that fixed point.
\end{theorem}
\begin{proof}
Let $x_0 \in X$ be arbitrary but fixed and consider the sequence of iterates $\{x_n\}$ where $x_n = Tx_{n-1}$ for all $n \geq 1$.

If $x_n = x_{n+1}$ for some natural number $n$, then it is easily noticeable that $T$ has a fixed point, the fixed point is unique and the sequence $\{x_n\}$ converges to that fixed point.

So now we presume that $x_n \neq x_{n+1}$ for all natural numbers $n$. Then arguing as in previous theorem we can show that all terms of {$\{x_n\}$} are distinct.

Now for any $n\in \mathbb{N}$ we have,
\begin{align*}
d(x_n,x_{n+1}) &< d(x_{n-1},x_n) \\
&<  d(x_{n-2},x_{n-1}) \\
& \hdots\\
&\hdots\\
&< d(x_0,x_1).
\end{align*}
Let $k,n \in \mathbb{N} $ be arbitrary but fixed. First suppose that $k\geq v$. Then by hypothesis we get a real number $M> 0$ such that $d(x_0,x_{k-v}) \leq M$. Therefore,
\begin{align*}
d(x_n,x_{n+k})&\leq  s\{d(x_n,x_{n+1}) + d(x_{n+1},x_{n+2}) + \hdots + d(x_{n+v},x_{n+k})\} \\
& <  s \{ d(x_0,x_1) + d(x_0,x_1) + \hdots + d(x_{n+v-1},x_{n+k-1})\} \\
& \hdots \\
& \hdots \\
& <  s \{ v d(x_0,x_1) + d(x_0,x_{k-v})\} \\
& <  s \{ v d(x_0,x_1) + M\}\\
& =  M_1, \mbox{  say }.
\end{align*}
Now suppose that $k<v$. Then we have
\begin{align} \label{eqn2}
d(x_n,x_{n+k})&\leq  s\{d(x_n,x_{n+1}) + d(x_{n+1},x_{n+2}) + \cdots + d(x_{n+v-1},x_{n+v})\nonumber \\ & + d(x_{n+v},x_{n+k})\} \nonumber \\
&< s\{d(x_0,x_1) + d(x_0,x_1) + \cdots + d(x_0,x_1) + d(x_{n+v-1},x_{n+k-1})\} \nonumber \\
& \hdots \nonumber \\
&\hdots \nonumber \\
&< s\{d(x_0,x_1) + d(x_0,x_1) + \cdots + d(x_0,x_1)+ d(x_v,x_k)\}.
\end{align}
Let $M_2 = \displaystyle \max_{k<v} \{d(x_v,x_k)\}$. Then clearly $M_2$ is finite. Thus by using Equation \ref{eqn2}, we get
$$ d(x_n,x_{n+k}) < s \{ v d(x_0,x_1) + M_2\} = M_3, \mbox{  say  }.$$
Therefore, from above discussions we see that 
$$ d(x_n,x_{n+k}) < \displaystyle \max \{M_1,M_3\}$$ for all $n,k\in \mathbb{N}$, which shows that the sequence $\{x_n\}$ is bounded. So by boundedly compactness of $X$, $\{x_n\}$ has a convergent subsequence, say, $\{x_{n_k}\}$ and also let $\displaystyle \lim_{k\to \infty} x_{n_k} = z.$ Then proceeding as Theorem \ref{th1} we can show that $z$ is the unique fixed point of $T$ and the sequence of iterates $\{x_n\}$ converges to $z$.
\end{proof}
As a consequence of Theorem \ref{th2} we obtain the following corollary.
\begin{corollary}
Let $(X,d)$ be a boundedly compact $b_v(s)$-metric space and $T:X\to X$ be a bounded mapping such that $$d(Tx,Ty) < d(x,y)$$ for all $x,y \in X$ with $x \neq y$. Then $T$ has a  unique fixed point and for any $x\in X$ the sequence  $\{T^nx\}$ converges to that fixed point.
\end{corollary}
Taking $s=1$ and $v=1$ in Theorem \ref{th2}, we get the following corollary.
\begin{corollary}
Let $(X,d)$ be a boundedly compact metric space and $T:X\to X$ be a mapping such that $$d(Tx,Ty)<d(x,y)$$ for all $x,y \in X$ with $x\neq y$. Also assume that, for any $x \in X$ and for any $k\in \mathbb{N}$ with $k\geq 1$ there exists a real number $M>0$ (depending on $x$) such that $d(x,T^{k-1}x) \leq M$. Then $T$ has a unique fixed point $z$ and for any $x\in X$ the sequence $\{T^nx\}$ converges to $z$.
\end{corollary}
Now we cite the following example which supports the above theorem.
\begin{example} \label{ex1}
Consider the set $X=\{\frac{1}{n}: n\in \mathbb{N}, n\geq 2\}$. We define a function $d:X\times X \to \mathbb{R}$ by
$$d(x,y)=\begin{cases}
|n-m|, ~~\text{if}~~ |n-m| \neq 1;\\
\frac{1}{2} ~~\text{if}~~ |n-m| =1.
\end{cases}$$
Then, it is an easy task to verify that $(X,d)$ is a $b_3(3)$-metric space. Also, it is simply noticeable that $(X,d)$ is  boundedly compact but not sequentially compact.

Now, we define a mapping $T:X \to X$ by
$$Tx= \frac{1}{4}, ~~\mbox{for all}~~ x\in X.$$
Then, clearly $$d(Tx,Ty) < d(x,y)$$ for all $x,y \in X$ with $x\neq y$.
%\begin{cases}
%\frac{1}{5}, ~~\text{if}~~ x=\frac{1}{2n} ~~ \mbox{for some}~~ n\in \mathbb{N};\\
%\frac{1}{4}, ~~\text{if}~~ x=\frac{1}{2n+1} ~~ \mbox{for some}~~ n\in \mathbb{N}.
%\end{cases}$$
%Let, $x,y \in X$ be arbitrary such that $x\neq y$, then the following cases may arise:
%Case I: If $x=\frac{1}{2n}$ and $y=\frac{1}{2m}$ for some $n,m \in \mathbb{N}$, then
%\begin{eqnarray*}
%d(Tx,Ty)=0;\\
%d(x,y) = 2|n-m|.
%\end{eqnarray*}
%So, it is clear that $d(Tx,Ty) < d(x,y)$.\\
%Case II: If $x=\frac{1}{2n+1}$ and $y=\frac{1}{2m+1}$ for some $n,m \in \mathbb{N}$, then
%\begin{eqnarray*}
%d(Tx,Ty)=0;\\
%d(x,y) = |2n+1-2m-1|=2|n-m|.
%\end{eqnarray*}
%Thus, it is easy to note that $d(Tx,Ty) < d(x,y)$.\\
%Case III: If $x=\frac{1}{2n}$ and $y=\frac{1}{2m+1}$ for some $n,m \in \mathbb{N}$, then
%\begin{eqnarray*}
%d(Tx,Ty)=d(\frac{1}{5},\frac{1}{4})=\frac{1}{2};\\
%d(x,y) = |2n-2m-1|.
%\end{eqnarray*}

Also, for any $x\in X$ and for any $k\in \mathbb{N}$ with $k \geq 3$, there exists a real number $M>0$ (depending on $x$) such that $d(x,T^{k-3}x)<M$ (here $M=x$).
Then by Theorem \ref{th2}, $T$ has a unique fixed point. Note that $x=\frac{1}{4}$ is the unique fixed point of $T$.
\end{example}

From the definitions of boundedly compactness and completeness of $b_v(s)$ metric space, we see that every boundedly compact $b_v(s)$-metric space is complete. Thus Example \ref{e1} also shows that if $(X,d)$ is a complete $b_v(s)$-complete metric space and $T:X\to X$ is a mapping such that $$d(Tx,Ty)<d(x,y)$$ for all $x,y \in$ with $x\neq y$, then $T$ may not have a fixed point. So we also need some additional assumption with the completeness of $X$ to warrant the fixed point of $T$.
\newpage
Now we consider the following example.
\begin{example}
Let $\{x_n^i\}$ be a sequence whose $i$-th term is $1$ and all other terms are $0$ and consider the set $$X=\left\{ \{x_n^i\} : i \in \mathbb{N} \right\}.$$
Now, we define a function $d:X \times X \to \mathbb{R}$ by
$$ d\left(\{x_n^i\},\{x_n^j\}\right)=
\begin{cases}
0, \text{ if}~~ i=j;\\
1 + \frac{100}{ \displaystyle \sum_{n=1}^{\infty}|i x_n^i - j x_n^j|} \text{ if } ~~ i,j \leq 10;\\
1 + \frac{10}{ \displaystyle \sum_{n=1}^{\infty}|i x_n^i - j x_n^j|} \text{ if } ~~ \mbox{either one of $i$ or $j$} >10.
\end{cases}$$
Then it is trivial to check that $(X,d)$ is a $b_2(10)$-metric space. Also,  $(X,d)$ is not boundedly compact but complete.

Define, $T: X \to X$ by $$ T\left(\{x_n^i\}\right) = \{x_n^{i+11}\} $$ for all $\{x_n^i\} \in X$.

Let, $\{x_n^i\}, \{x_n^j\} \in X$ be arbitrary with $i=j$, then the following three cases may arise:

Case I: $i,j \leq 10$, then
\begin{align*}
d\left(T \{x_n^i\},T \{x_n^j\}\right) &= d\left( \{x_n^{i+11}\}, \{x_n^{j+11}\}\right)\\
&= \frac{10}{i+j+22}+1\\
&< \frac{100}{i+j}+1\\
&= d\left(\{x_n^i\},\{x_n^j\}\right).
\end{align*}
Case II: $i,j> 10$, then
\begin{align*}
d\left(T \{x_n^i\},T \{x_n^j\}\right) &= d\left( \{x_n^{i+11}\}, \{x_n^{j+11}\}\right)\\
&= \frac{10}{i+j+22}+1\\
&< \frac{10}{i+j}+1\\
&= d\left(\{x_n^i\},\{x_n^j\}\right).
\end{align*}
Case III: any one of $i$ and $j >10$, then
\begin{align*}
d\left(T \{x_n^i\},T \{x_n^j\}\right) &= d\left( \{x_n^{i+11}\}, \{x_n^{j+11}\}\right)\\
&= \frac{10}{i+j+22}+1\\
&< \frac{10}{i+j}+1\\
&= d\left(\{x_n^i\},\{x_n^j\}\right).
\end{align*}
Thus, $$d(Tx,Ty)< d(x,y)$$ for all $x,y \in X$ with $x\neq y$.

Note that for any $x \in X$ and for any $k\in \mathbb{N}$ with $k\geq v$(here $v=2$) there exists a real number $M>0$ (here we may take $M=200$) such that $d(x,T^{k-v}x) \leq M$ but still $T$ has no fixed point.
\end{example}
Thus we see that the additional condition which is considered in Theorem \ref{th2} together with the completeness of $X$ does not deliver fixed point of $T$. It means that we have to find out some different additional condition with the completeness of $X$ so as to ensure the existence of  fixed point of $T$. In the following theorem we consider such an addition condition in case of $b_v(1)$-metric spaces which is firstly given by Suzuki \cite{S2}. 
\begin{theorem} \label{th3}
Let $(X,d)$ be a complete $b_v(1)$-metric space and $T:X\to X$ be a mapping such that $$d(Tx,Ty) < d(x,y)$$ for all $x,y \in X$ with $x \neq y$. Also assume that for any $x\in X$ and for any $\epsilon >0$, there exists $\delta >0$ such that $$d(T^ix,T^jx)< \epsilon + \delta  \Rightarrow d(T^{i+1}x,T^{j+1}x)\leq \epsilon$$ for any $i,j \in \mathbb{N} \cup \{0\}$. Then $T$ has a unique fixed point and for any $x\in X$ the sequence  $\{T^nx\}$ converges to that fixed point.
\end{theorem}
\begin{proof}
Choose an arbitrary but fixed element $x_0 \in X$, and consider the sequence $\{x_n\}$ in $X$, where $x_n=T^nx_0$ for all natural numbers $n$.

If $x_n = x_{n+1}$ for some natural number $n$, then it is easy to check that $x_n$ is the unique fixed point of $T$. So, now we assume that no two consecutive terms of $\{x_n\}$ are equal, i.e., $x_n \neq x_{n+1}$ for every natural number $n$.

Under this assumption, proceeding as Theorem \ref{th1}, we can show that all terms of $\{x_n\}$ are distinct.

Consider the sequence of real numbers $\{s_n\}$ where $s_n=d(x_n,x_{n+1})$ for all $n\in \mathbb{N}$. Then,
\begin{align*}
s_{n+1}&= d(x_{n+1},x_{n+2})\\
&< d(x_n,x_{n+1})\\
&= s_n
\end{align*}
for all $n \in \mathbb{N}$.

Therefore, $\{s_n\}$ is a decreasing sequence of non-negative real numbers and hence convergent to some $a\geq 0$. If $a>0$, then by given condition there exists $\delta>0$ such that 
$$d(x_n,x_{n+1})<a+\delta ~~\implies ~~ d(x_{n+1},x_{n+2})\leq a$$ for all $n \in \mathbb{N}$.

Again by definition of $a$, for the above $\delta>0$, there exists $n\in \mathbb{N}$ such that
$$d(x_n,x_{n+1})<a+\delta.$$

Therefore, $d(x_{n+1},x_{n+2})\leq a$ and this leads to a contradiction. so we must have $a=0$, i.e.,
\begin{equation} \label{eqq1}
\displaystyle \lim_{n \to \infty} d(x_n,x_{n+1})=0.
\end{equation} 
In a similar manner, we can show that 
\begin{equation} \label{eqq2}
\displaystyle \lim_{n \to \infty} d(x_n,x_{n+2})=0.
\end{equation}
Now let, $\epsilon >0$ be arbitrary. Then by given condition we can obtain a $\delta>0$ such that
$$d(x_n,x_{n+1})<\frac{\epsilon}{2}+\delta ~~\mbox{implies} ~~ d(x_{n+1},x_{n+2})\leq \frac{\epsilon}{2}$$ for all $n \in \mathbb{N}$.

Also by Equations \ref{eqq1} and \ref{eqq2}, for the above $\delta>0$ there exists a natural number $N$ such that
\begin{align}
d(x_n,x_{n+1}) &< \frac{\delta}{2v}  \label{eqq3} \\
d(x_n,x_{n+2}) &< \frac{\delta}{2v} \label{eqq4}
\end{align}
for all $n\geq N$.

Now we show by induction on $j$ that $d(x_{n+2v+1},x_{n+2v+j}) \leq \frac{\epsilon}{2}$ for all $j\in \mathbb{N}$ and for all $n \geq N$.

The result is obviously true for $j=1$. Assume that, the result is true for some $j\in \mathbb{N}$. So,
$d(x_{n+2v+1},x_{n+2v+j}) \leq \frac{\epsilon}{2}$ which implies
\begin{equation} \label{eqq5}
d(x_{n+3v+1},x_{n+3v+j}) \leq \frac{\epsilon}{2}.
\end{equation}

Now,
\begin{align}
d(x_{n+2v},x_{n+2v+j}) & \leq  \{d(x_{n+2v},x_{n+2v+2})+d(x_{n+2v+2},x_{n+2v+3}) +d(x_{n+2v+3},x_{n+2v+4})\nonumber \\    
& + \hdots +d(x_{n+2v+v},x_{n+2v+v+1})
+d(x_{n+2v+v+1},x_{n+2v+j}) \}\nonumber \\ \nonumber
&= \{d(x_{n+2v},x_{n+2v+2})+d(x_{n+2v+2},x_{n+2v+3}) +d(x_{n+2v+3},x_{n+2v+4})\\ \nonumber     
& + \hdots +d(x_{n+3v},x_{n+3v+1})\}+d(x_{n+2v+j},x_{n+3v+1}) \\ \nonumber
&\leq   \{d(x_{n+2v},x_{n+2v+2})+d(x_{n+2v+2},x_{n+2v+3}) +d(x_{n+2v+3},x_{n+2v+4})\\ \nonumber    
& + \dots +d(x_{n+3v},x_{n+3v+1})\} + \{d(x_{n+2v+j},x_{n+2v+j+1})  \nonumber\\
&  +(x_{n+2v+j+1},x_{n+2v+j+2}) + \dots +d(x_{n+2v+j+v-1},x_{n+2v+j+v}) \nonumber \\
& + d(x_{n+2v+j+v},x_{n+3v+1}) \} \nonumber\\ \nonumber
& = \{d(x_{n+2v},x_{n+2v+2})+d(x_{n+2v+2},x_{n+2v+3}) +d(x_{n+2v+3},x_{n+2v+4}) \\  \nonumber   
& + \dots +d(x_{n+3v},x_{n+3v+1})\} + \{d(x_{n+2v+j},x_{n+2v+j+1}) \\
& + d(x_{n+2v+j+1},x_{n+2v+j+2}) + \dots +d(x_{n+3v+j-1},x_{n+3v+j})\} \nonumber \\
& + d(x_{n+3v+j},x_{n+3v+1}).  \label{eqq6}
\end{align}
Using Equations \ref{eqq3}, \ref{eqq4} and \ref{eqq5} in equation \ref{eqq6} we get
\begin{align*}
d(x_{n+2v},x_{n+2v+j}) &< \left( \frac{\delta}{2v}+\frac{\delta}{2v} + \cdots +\frac{\delta}{2v}\right) + \left( \frac{\delta}{2v}+\frac{\delta}{2v} + \cdots +\frac{\delta}{2v}\right ) + \frac{\epsilon}{2}\\
\Rightarrow d(x_{n+2v},x_{n+2v+j}) &< \frac{\epsilon}{2} + \delta
\end{align*}
for all $n \geq N$. This implies that $$d(x_{n+2v+1},x_{n+2v+j+1}) \leq \frac{\epsilon}{2} $$ for all $n \geq N$. Thus by induction it follows that 
$$d(x_{n+2v+1},x_{n+2v+j})\leq \frac{\epsilon}{2} < \epsilon$$ for all $n \geq N$ and for all $j \in \mathbb{N}$, which proves that $\{x_n\}$ is a Cauchy sequence. So, by the completeness of $X$, $\{x_n\}$ is convergent to some $z\in X$.

Now continuing as Theorem \ref{th1}, we can prove that $z$ is the only fixed point of $T$.

Since, we choose $x_0 \in X$ arbitrarily, it follows that $\{T^nx\}$ converges to the unique fixed point $z$ for all $x\in X$.

\end{proof}
\begin{open question}
Let $(X,d)$ be a complete $b_v(s)$-metric space and $T$ be a self-mapping on $X$ such that $$d(Tx,Ty)<d(x,y)$$ for all $x,y \in X$ with $x\neq y$. If $s>1$, then find out a weaker additional assumption on $T$ which will ensure that $T$ has a fixed point.
\end{open question}

From Theorem \ref{th3} we can derive several important corollaries. We present a number of selected ones which extend several well-known results in the literature.
%If we take  $v=1$ in Theorem \ref{th3}, we get the following corollary (this is the main result in \cite{S2}):
\begin{corollary} [Theorem $5$, \cite{S2}]
Let $(X,d)$ be a complete metric space and $T:X\to X$ be a mapping such that $$d(Tx,Ty) < d(x,y)$$ for all $x,y \in X$ with $x \neq y$. Also assume that for any $x\in X$ and for any $\epsilon >0$, there exists $\delta >0$ such that $$d(T^ix,T^jx)< \epsilon + \delta  \Rightarrow d(T^{i+1}x,T^{j+1}x)\leq \epsilon$$ for any $i,j \in \mathbb{N} \cup \{0\}$. Then $T$ has a unique fixed point and for any $x\in X$ the sequence  $\{T^nx\}$ converges to that fixed point.
\end{corollary}
%If we take  $v=2$ in Theorem \ref{th3}, we get the following corollary:
\begin{corollary}
Let $(X,d)$ be a complete rectangular metric space and $T:X\to X$ be a mapping such that $$d(Tx,Ty) < d(x,y)$$ for all $x,y \in X$ with $x \neq y$. Also assume that for any $x\in X$ and for any $\epsilon >0$, there exists $\delta >0$ such that $$d(T^ix,T^jx)< \epsilon + \delta  \Rightarrow d(T^{i+1}x,T^{j+1}x)\leq \epsilon$$ for any $i,j \in \mathbb{N} \cup \{0\}$. Then $T$ has a unique fixed point and for any $x\in X$ the sequence  $\{T^nx\}$ converges to that fixed point.
\end{corollary}
\vskip.5cm\noindent{\bf Acknowledgements:}\\
The first named author would like to express his special thanks of gratitude to  CSIR, New Delhi, India for their financial supports. 
\bibliographystyle{plain}
%\bibliography{bibliography_file}

\end{document}